\documentclass{amsart}
\usepackage{amssymb,latexsym}
\usepackage{graphicx, pgf}
\theoremstyle{plain}

\newtheorem{lemma}{Lemma}

\newtheorem*{theo}{Theorem}

\theoremstyle{definition}

\newtheorem*{defin}{Definition}

\theoremstyle{remark}

\newtheorem{remark}{\bf Remark}
\numberwithin{equation}{section}
\newcommand{\bs}{\boldsymbol}

\begin{document}


\title[Thermistor system with $p$-Laplacian]
      {Existence of weak solutions of an unsteady thermistor system with $p$-Laplacian type equation}
\author{Joachim Naumann}
\address{Mathematics Department\\
         Humboldt University Berlin\\
         Unter den Linden 6\\
         D-10099 Berlin} 
\email{jnaumann@math.hu-berlin.de}
\date{September 1, 2015}

\begin{abstract}
In this paper, we consider an unsteady thermistor system, where the usual Ohm law is replaced by a non-linear monotone constitutive relation between current and electric field. This relation is modeled by a $p$-Laplacian type equation for the electrostatic potential $\varphi$. We prove the existence of weak solutions of this system of PDEs under mixed boundary conditions for $\varphi$, and a Robin boundary condition and an initial condition for the temperature $u$. 
\end{abstract}

\maketitle

\section{Introduction}\label{intro} 
Let $\Omega\subset\mathbb{R}^n$ ($n=2$ or $n=3$) be a bounded domain with Lipschitz boundary $\partial\Omega$, and set $Q_T=\Omega\times\,]\,0,T\,[\,$ ($0<T<+\infty$).\par
Let $\bs{J}$ and $\bs{q}$ denote the electric current field density and the heat flux, respectively, of a thermistor occupying the domain $\Omega$ under unsteady operating conditions. Then the balance equations for the electric current and the heat flow within the thermistor material are the following two PDEs
\[
 \nabla\cdot\boldsymbol{J}=0,\quad \frac{\partial u}{\partial  t}+\nabla\cdot\bs{q}=f(x,t,u,\nabla\varphi)\quad\text{in }\; Q_T,
\]
where $\varphi=\varphi(x,t)$ and $u=u(x,t)$ represent the electrostatic potential and the temperature, respectively 
(see, e.g., \cite[Chap. 8]{25}).\par
We make the following constitutive assumptions on $\bs{J}$ and $\bs{q}$
\[
 \bs{J}=\sigma\big(u,|\bs{E}|\big)\bs{E}\quad\text{ Ohm's law},\quad  \bs{q}=-\kappa(u)\nabla u\quad\text{ Fourier's law},
\]
where
\begin{align*}
 \bs{E}&=-\nabla\varphi\quad\text{ density of the electric field},\\
 \sigma&=\sigma\big(u,|\bs{E}|\big)\quad\text{ electrical conductivity},\\
 \kappa&=\kappa(u)\quad\text{ thermal conductivity}.
\end{align*}
With these notations the above system of PDEs takes the form
\begin{align}\label{1.1}
 -\nabla\cdot\big(\sigma\big( u,|\nabla\varphi|\big)\nabla\varphi\big)&=0\quad\text{ in }\;Q_T,\\
 \label{1.2}\frac{\partial u}{\partial t}-\nabla\cdot\big(\kappa(u)\nabla u\big)&=f(x,t,u,\nabla\varphi)\quad\text{ in }\;Q_T.
\end{align}
The function $f=f(x,t,u,\nabla\varphi)$ represents a heat source that will be specified below (see (\ref{1.13}) and (H3), Section~2). 
\par
We supplement system (\ref{1.1})--(\ref{1.2}) by boundary conditions for $\varphi$ and $u$, and an initial condition for $u$. Without any further reference, throughout the paper we assume\vspace*{1mm}
\[
 \partial\Omega=\Gamma_D\cup\Gamma_N\;\text{ disjoint}, \quad\Gamma_D\;\text{ non-empty, open}.
\]
Define
\[
 \Sigma_D=\Gamma_D\times\,]0,T\,[\,,\quad\Sigma_N=\Gamma_N\times\,]\,0,T\,[\,.
\]
We then consider the conditions
\begin{eqnarray}\label{1.3}
&&\varphi=\varphi_D \ \text{ on } \ \Sigma_D,\quad \bs{J}\cdot\bs{n}=0 \ \text{ on } \ \Sigma_N,\\[1mm]
\label{1.4}
&&\bs{q}\cdot\bs{n}= g(u-h) \ \text{ on } \ \partial\Omega\times\,]\,0,T\,[\,,\\[1mm]
\label{1.5}
&&u=u_0 \ \text{ in } \ \Omega\times\{0\}
\end{eqnarray}\vspace*{-3mm}

\noindent
($\bs{n}=$~unit outward normal to $\partial\Omega$). The first condition in (\ref{1.3}) means that there is an applied voltage $\varphi_D$ along $\Sigma_D$, whereas the second condition characterizes electrical insulation of the thermistor along $\Sigma_N$. The Robin boundary condition (\ref{1.4})\footnote{This boundary condition is also called ``Newton's cooling law'' or ``third boundary condition''.} means that the flux of heat through $\partial\Omega\times\,]\,0,T\,[\,$ is proportional to the temperature difference $u-h$, where $g$ denotes the thermal conductivity of the surface $\partial\Omega$ of the thermistor, and $h$ represents the ambient temperature (cf. \cite{7}, \cite{9}, \cite{11}, \cite{18}, \cite[Chap. 8]{25} and \cite{28} (nonlinear boundary conditions)).\\[0.05mm]
\mbox{}\hfill$\square$

We present two prototypes for the electrical conductivity $\sigma$. To this end, let $\sigma_0:\mathbb{R}\to\mathbb{R}_+$\footnote{$\mathbb{R}_+=[0,+\infty\,[\,$.} be a continuous function such that 
\[
 0<\sigma_*\le\sigma(u)\le\sigma^*<\infty\quad\forall\; u\in\mathbb{R}\quad(\sigma_*,\sigma^*=\mathrm{const}).
\]
We then consider the following functions
\begin{equation}\label{1.6}
\sigma(u,\tau)=\sigma_0(u)(\delta+\tau^2)^{(p-2)/2},\quad (u,\tau)\in\mathbb{R}\times\mathbb{R}_+\quad
         (\delta=\mathrm{const} >0,\;\:1<p<+\infty)
\end{equation}
and 
\begin{equation}\label{1.7}
 \sigma(u,\tau)=\sigma_0(u)\tau^{p-2},\quad (u,\tau)\in\mathbb{R}\times\mathbb{R}_+\quad (2\le p<+\infty).
\end{equation}
The electrical conductivities which correspond to these functions $\sigma=\sigma(u,\tau)$ read
\begin{equation}\label{1.8}
 \sigma\big(u,|\bs{E}|\big)=\sigma_0(u)\big(\delta+|\bs{E}|^2\big)^{(p-2)/2}
\end{equation}
and
\begin{equation}\label{1.9}
 \sigma\big(u,|\bs{E}|\big)=\sigma_0(u)|\bs{E}|^{p-2},
\end{equation}
respectively ($\bs{E}=$~electrical field density). Here, the factor $\sigma_0(u)$ characterizes the thermal dependence of the electrical conductivity of the thermistor material. Observing that $\bs{E}=-\nabla\varphi$, equ. (\ref{1.1}) takes the form of $p$-Laplacian equations
\[
 -\nabla\cdot\big(\sigma_0(u)\big(\delta+|\nabla\varphi|^2\big)^{(p-2)/2}\nabla\varphi\big)=0,
\]
resp.
\[
 -\nabla\cdot\big(\sigma_0(u)|\nabla\varphi|^{p-2}\nabla\varphi\big)=0.
\]
\\[0.05mm]
Let $p=2$. Then both (\ref{1.8}) and (\ref{1.9}) lead to $\bs{J}=\sigma_0(u)\bs{E}$. If the right hand side in (\ref{1.2}) is of the form 
\[
 f=\sigma_0(u)|\nabla\varphi|^2=\bs{J}\cdot\bs{E}\quad (\text{Joule heat})
\]
(cf. (\ref{1.13}) below), then (\ref{1.1})--(\ref{1.2}) represents the ``classical'' thermistor system (see \cite{1}, \cite{8}, \cite{11}, \cite{29}). This system has been studied in \cite{14}--\cite{16} with a degeneration of the coefficients $\sigma_0(u)$ and $\kappa(u)$ (cf. also \cite{9} for a similar degeneration of $\sigma_0(u)$).\hfill$\square$

\begin{remark}\label{r1}
({\it The case\/} $1<p\le 2$.) Let be $\sigma=\sigma(u,\tau)$ as in \textrm{(\ref{1.6})}. Then Ohm's law reads
\begin{equation}\label{1.10}
 \bs{J}=\sigma_0(u)\big(\delta+|\bs{E}|^2\big)^{(p-2)/2}\bs{E}
\end{equation}
(cf. (\ref{1.8})). To make things clearer, let $I=|\bs{J}|$ and $V=|\bs{E}|$ denote the current and voltage, respectively, in an electrical conductor. Equ. (\ref{1.10}) then gives the current-voltage characteristic
\begin{equation}\label{1.11}
 I=\sigma_0(u)(\delta+V^2)^{(p-2)/2}V.
\end{equation}
\end{remark}

If $p$ is ``sufficiently near to $1$'', then (\ref{1.11}) can be used as an approximation of current-voltage characteristics for transistors (see, e.g., \cite{19}, \cite[Chap. 6.2.2]{27}).\par
The characteristic (\ref{1.11}) continues to make sense if $p=1$, i.e.,
\begin{equation}\label{1.12}
 I=\frac{\sigma_0(u)}{(\delta+V^2)^{1/2}}\,V.
\end{equation}
This current-voltage characteristic is widely used to describe the effect of saturation of current in certain transistors under high electric fields (see, e.g., \cite[Chap. 2.5]{23} for details). 
\par
Finally, we notice that for the case $\delta=0$ and $p=1$, Ohm's law (\ref{1.10}) and the current-voltage characteristic (\ref{1.11}) have to be replaced by
\begin{eqnarray*}
 \bs{J}\in\overline{B_{r_0}(0)}\quad\text{ if}\quad \bs{E}=\bs{0},&&\bs{J}=\frac{r_0}{|\bs{E}|}\,\bs{E}\quad\text{ if}\quad \bs{E}\ne\bs{0},\\
 0\le I\le r_0\quad\text{ if}\quad V=0,&& I=r_0\quad\text{ if}\quad V>0,
 \end{eqnarray*}
respectively, where $\overline{B_{r_0}(0)}=\{\xi\in\mathbb{R}^n; |\xi|\le r_0\}$, $r_0=r_0(u)$ (cf. \cite{17}).

\begin{remark}\label{r2}
({\it The case\/} $2\le p<+\infty$.) In \cite{20}, the authors consider current-voltage characteristics of the form 
\[
I=\big(\sigma_0(x,u)V^{p(x)-2}\big)V=\sigma_0(x,u)V^{p(x)-1},\quad 2\le p(x)<+\infty,\; x\in\Omega
 \]
(cf. (\ref{1.7}) and (\ref{1.9})). These characteristics are used to model organic semiconductors, where $p=p(x)$ is a jump function that characterizes Ohmic and non-Ohmic contacts of the device material.\hfill$\square$ 
\end{remark}

We present a prototype for the heat source term $f$ in (\ref{1.2}) which motivates hypotheses (H3) in Section 2.
\par
Let be $\sigma=\sigma(u,\tau)$ as in (\ref{1.6}) or (\ref{1.7}). For $\big((x,t),u,\xi\big)\in Q_T\times\mathbb{R}\times\mathbb{R}^n$ we consider functions $f$ as follows
\begin{equation}\label{1.13}
 \left\{ \begin{array}{l}
          f(x,t,u,\xi)=\alpha(x,t,u,\xi)\sigma\big(u,|\xi|\big)|\xi|^2,\\[1mm]
          \alpha:Q_T\times\mathbb{R}\times\mathbb{R}^n\longrightarrow\mathbb{R}_+\quad\text{is Carath\'eodory},\\[1mm]
          0\le\alpha(x,t,u,\xi)\le\alpha_0=\mathrm{const}\quad\forall\;\big((x,t),u,\xi\big)\in Q_T\times\mathbb{R}\times\mathbb{R}^n\quad (\alpha_0=\mathrm{const}).
         \end{array}\right.
\end{equation}
If $\alpha\equiv 1$, with the above notations $\bs{J}$ and $\bs{E}$ we obtain the source term 
\[
 f(x,t,u,\nabla\varphi)=\sigma\big(u,|\nabla\varphi|\big)(-\nabla\varphi)\cdot(-\nabla\varphi)=\bs{J}\cdot\bs{E}.
\]

Let be $\alpha$ of the form
\[
 \alpha(x,t,u,\xi)=\widehat{\alpha}(x,t,u,-\xi)
\]
or
\[
 \alpha(x,t,u,\xi)=\widehat{\alpha}\big(x,t,u,-\sigma\big(u,|\xi|\big)\xi\big),
\]
where $\widehat{\alpha}:Q_T\times\mathbb{R}\times\mathbb{R}^n$ is a Carath\'eodory function such that $0\le \widehat{\alpha}\le 1$ everywhere. Then (\ref{1.2}) models a self-heating process with source term 
\[
 f=\alpha\,\bs{J}\cdot\bs{E},
\]
where the factor
\[
 \alpha=\widehat{\alpha}(x,t,u,\bs{E})\quad \text{or}\quad \alpha=\widehat{\alpha}(x,t,u,\bs{J}),
\]
respectively, characterizes a loss of Joule heat (cf. \cite{20} 
for more details).
\smallskip

The existence of weak solutions to the {\it steady case\/} of (\ref{1.1})--(\ref{1.4}) has been proved for the first time in \cite{20} for $2<p<+\infty$ and in \cite{13} for $2\le p(x)<+\infty$ ($n=2$ in both papers). Extensions of these results have been recently presented in \cite{5}, \cite{6}.\hfill$\square$
\smallskip

In \cite{24}, we proved the existence of a weak solution of (\ref{1.1})--(\ref{1.5}) when the function $\tau\mapsto\sigma(u,\tau)$ is strictly monotone and $f$ satisfies hypothesis (H3) below (see Section~2) which includes (\ref{1.13}) as a special case. The aim of the present paper is to prove an analogous existence result when $\tau\mapsto\sigma(u,\tau)$ is merely monotone whereas the function $f$, however, has to satisfy a structure condition of type (\ref{1.13}).

\section{Weak formulation of (\ref{1.1})--(\ref{1.5})}\label{s2} 

We introduce the notations which will be used in what follows.
\par
By $W^{1,p}(\Omega)$ ($1\le p<+\infty$) we denote the usual Sobolev space. Define
\[
 W_{\Gamma_D}^{1,p}(\Omega)=\big\{ v\in W^{1,p}(\Omega); v=0 \ \text{ a.e. on } \ \Gamma_D\big\}.
\]
This space is a closed subspace of $W^{1,p}(\Omega)$. Throughout the paper, we consider $W_{\Gamma_D}^{1,p}(\Omega)$ equipped with the norm
\[                                                                                                                                             |v|_{W^{1,p}}=\left(\,\int\limits_\Omega|\nabla v|^p dx\right)^{1/p}.                                                                                                                                            \]

Let $X$ denote a real normed space with norm $|\cdot|_X$ and let $X^*$ be its dual space. By $\langle x^*,x\rangle_X$ we denote the dual pairing between $x^*\in X^*$ and $x\in X$. The symbol $L^p(0,T,X)$ ($1\le p\le+\infty$) stands for the vector space of all strongly measurable mappings $u:\,]\,0,T\,[\,\to X$ such that the function $t\mapsto\big|u(t)\big|_X$ is in $L^p(0,T)$ (cf. \cite[Chap. III, \S3; Chap. IV, \S3]{2}, \cite[App.]{3}, \cite[Chap. 1]{10}). For $1\le p<+\infty$, the spaces $L^p\big(0,T;L^p(\Omega)\big)$ and $L^p(Q_T)$ are linearly isometric. Therefore, in what follows we identify these spaces. \par
Let $H$ be a real Hilbert space with scalar product $(\cdot,\cdot)_H$ such that $X\subset H$ densely and continuously. Identifying $H$ with its dual space $H^*$ via Riesz' Representation Theorem, we obtain the continuous embedding $H\subset X^*$ and 
\begin{equation}\label{2.1}
 \langle h,x\rangle_X=(h,x)_H\quad\forall\; h\in H, \ \forall\; x\in X.
\end{equation}
Given any $u\in L^1(0,T;X)$ we identify this function with a function in $L^1(0,T;X^*)$ and denote it again by $u$. If there exists $U\in L^1(0,T;X^*)$ such that 
\[
 \int\limits_0^T u(t)\alpha'(t)dt\mathop{=}\limits^{\mathrm{in } X^*}-\int\limits_0^T U(t)\alpha(t)dt\quad\forall\;\alpha\in C_c^\infty(\,]\,0,T\,[\,),
\]
then $U$ will be called derivative of $u$ in the sense of distributions from $\,]\,0,T\,[\,$ into $X^*$ and denoted by $u'$ (see \cite[App.]{3}, \cite[Chap. 21]{10}).\hfill$\square$
\par
Let $1<p<+\infty$ be fixed. We make the following assumptions on the coefficients $\sigma$, $\kappa$ and the right hand side $f$ in (\ref{1.1})--(\ref{1.2}):
\[
\begin{array}{l}
\text{(H1)}\qquad\left\{\begin{array}{l}
           \sigma:\mathbb{R}\times\mathbb{R}_+\to\mathbb{R}_+ \ \text{ is continuous},\\[1mm]
           c_1\tau^p-c_2\le\sigma(u,\tau)\tau^2,\; 0\le\sigma(u,\tau)\le c_3(1+\tau^2)^{(p-2)/2}\\[1mm]
           \forall\;(u,\tau)\in\mathbb{R}\times\mathbb{R}_+, \text{ where } c_1,c_3=\mathrm{const}>0\text{ and } c_2=\mathrm{const}\ge0;
                        \end{array}\right.\\[7mm]
 \text{(H2)}\qquad \left\{\begin{array}{l}
                           \kappa:\mathbb{R}\to\mathbb{R}_+ \ \text{ is continuous},\\[1mm]
                           0 <\kappa_0\le\kappa(u)\le\kappa_1 \quad\forall\; u\in\mathbb{R}, \text{ where } \ \kappa_0,\kappa_1=\mathrm{const},
                          \end{array}\right.
                          \end{array}
\]
and 
\[
 \text{(H3)}\qquad\left\{\begin{array}{l}
                   f:Q_T\times\mathbb{R}\times\mathbb{R}^n\to\mathbb{R}_+ \ \text{ is Carath\'eodory}, \\[1mm] 
                   0\le f(x,t,u,\xi)\le c_4\big(1+|\xi|^p\big)\\[1mm]
                   \forall\; (x,t,u,\xi)\in Q_T\times\mathbb{R}\times\mathbb{R}^n,\text{ where } \ c_4=\mathrm{const} >0.\qquad\qquad\quad
                        \end{array}\right.
\]
It is readily seen that (H1) and (H3) are satisfied by the prototypes for $\sigma$ and $f$ we have considered in Section 1.\hfill$\square$

\begin{defin}
 Assume (H1)--(H3) and suppose that the data in (\ref{1.3})--(\ref{1.5}) satisfy
 \begin{eqnarray}\label{2.2}
&& \hspace*{-4cm} \varphi_D\in L^p\big(0,T;W^{1,p}(\Omega)\big);\\[1mm]
 \label{2.3}
&& \hspace*{-4cm} g=\mathrm{const}, \quad h=\mathrm{const};\\[1mm]
\label{2.4}
&&\hspace*{-4cm} u_0\in L^1(\Omega).
\end{eqnarray}
The pair
\[
(\varphi,u)\in L^p\big(0,T;W^{1,p}(\Omega)\big)\times L^q\big(0,T;W^{1,q}(\Omega)\big)\quad \Big(1<q<\frac{n+2}{n+1}\Big)
\]
is called {\it weak solution\/} of (\ref{1.1})--(\ref{1.5}) if
\begin{eqnarray}\label{2.5}
&& \int\limits_{Q_T}\sigma\big(u,|\nabla\varphi|\big)\nabla\varphi\cdot\nabla\zeta\, dxdt=0\quad\forall\;\zeta\in L^p\big(0,T;W_{\Gamma_D}^{1,p}(\Omega)\big);\\[1mm]
\label{2.6}
&& \varphi=\varphi_D\quad\text{a.e. on }\; \Sigma_D;\\[1mm]
\label{2.7}
&& \exists\;u'\in L^1\big(0,T;\big(W^{1,q'}(\Omega)\big)^*\big);\\[1mm]
\label{2.8}
&&\left\{\begin{array}{l}
        {\displaystyle\int\limits_0^T\big\langle u'(t),v(t)\big\rangle_{W^{1,q'}}dt+\int\limits_{Q_T}\kappa(u)\nabla u\cdot\nabla v\, dxdt+g\int\limits_0^T\int\limits_{\partial\Omega}(u-h)v\,d_x Sdt}\\
        ={\displaystyle \int\limits_{Q_T}f(x,t,u,\nabla\varphi)v\,dxdt\quad\forall\; v\in L^\infty\big(0,T;W^{1,q'}(\Omega)\big);}
       \end{array}\right.\\[1mm]
\label{2.9}
&& u(0)=u_0\quad\text{ in } \ \big(W^{1,q'}(\Omega)\big)^*.
\end{eqnarray}
\end{defin}

From (H1) and (H3) it follows that $f(\cdot,\cdot,u,\nabla\varphi)\!\in\! L^1(Q_T)$. Therefore, 
$u\!\in\! L^q\big(0,T;W^{1,q}(\Omega)\big)$ $\big(1<q<\frac{n+2}{n+1}\big)$ is standard for weak solutions of parabolic equations with right hand side in $L^1$ (see, e.g., the papers cited in \cite{24}).
\par
We notice that $v\in L^\infty\big(0,T,W^{1,q'}(\Omega)\big)$ can be identified with a function in $L^\infty(Q_T)$ (cf. \cite{24}). Hence, the integral on the right hand side of the variational identity in (\ref{2.8}) is well-defined.
\smallskip

To make precise the meaning of (\ref{2.9}), let $\frac{2n}{n+2}<q<\frac{n+2}{n+1}$. Then $\frac{nq}{n-q}>2$ and $q'> n+2$. Identifying $L^2(\Omega)$ with its dual, we obtain
\begin{eqnarray}\label{2.10}
& W^{1,q'}(\Omega)\;\subset \;W^{1,q}(\Omega)\;\subset\; L^2(\Omega)\;\subset\;\big(W^{1,q'}(\Omega))^*.&\\
&\text{\small continuously} \quad\text{\small compactly}\quad \text{\small continuously}&\nonumber
\end{eqnarray}
Therefore, $u$ can be identified with an element in $L^q\big(0,T;\big(W^{1,q'}(\Omega)\big)^*\big)$. Together with (\ref{2.7}) this implies the existence of a function $\widetilde{u}\in C\big([0,T];\big( W^{1,q'}(\Omega)\big)^*\big)$ such that
\[
 \widetilde{u}(t)=u(t)\quad\text{for a.e. }\; t\in [0,T]
\]
(see, e.g., \cite[p. 45, Th. 2.2.1]{10}).
\par
On the other hand, there exists a uniquely determined $\widetilde{u}_0\in \big(W^{1,q'}(\Omega)\big)^*$ such that 
\begin{equation}\label{2.11}
 \langle \widetilde{u}_0,z\rangle_{W^{1,q'}}=\int\limits_\Omega u_0z\,dx\quad\forall\; z\in W^{1,q'}(\Omega).
\end{equation}
Thus, (\ref{2.9}) has to be understood in the sense
\[
 \widetilde{u}(0)=\widetilde{u}_0\quad\text{ in }\;\:\big(W^{1,q'}(\Omega)\big)^*.
\]

\begin{remark}\label{r3}
 Let $(\varphi,u)$ be a sufficiently regular solution of (\ref{1.1})--(\ref{1.5}). We multiply (\ref{1.1}) and (\ref{1.2}) by smooth test functions $\zeta$ and $v$, respectively, satisfying the conditions
 \[
  \zeta=0\quad\text{on }\; \Sigma_D, \quad v(\cdot,T)=0 \quad\text{in }\;\Omega.  
 \]
Then we integrate the div-terms by parts over $\Omega$ and the term $\frac{\partial u}{\partial t}v$ by parts over the interval $[0,T]$. It follows
\begin{eqnarray}\label{2.12}
&&-\int\limits_{Q_T}u\,\frac{\partial v}{\partial t}\,dxdt+\int\limits_{Q_T} \kappa(u)\nabla u\cdot\nabla v \,dxdt+g\int\limits_0^T \int\limits_{\partial\Omega}(u-h)v\,d_xSdt\nonumber\\
&&=\int\limits_\Omega u_0v(\cdot,0)dx+\int\limits_{Q_T}f(x,t,u,\nabla\varphi)v\, dxdt.
\end{eqnarray}
This variational formulation of initial/boundary-value problems for parabolic equations is frequently used in the literature. 
\par
We notice that from a variational identity of type (\ref{2.12}) it follows the existence of a distributional time derivative of $u$ (see the arguments concerning (\ref{4.25}) and (\ref{4.26}) below).
\end{remark}

\begin{remark}\label{r4}
 Let $(\varphi,u)$ be a weak solution of (\ref{1.1})--(\ref{1.5}). From (\ref{2.8}) it follows that, for any $z\in W^{1,q'}(\Omega)$,
\begin{eqnarray}\label{2.13}
&& \big\langle u'(t),z\big\rangle_{W^{1,q'}}+\int\limits_\Omega \kappa\big(u(x,t)\big) \nabla u(x,t)\cdot\nabla z(x)dx+g\int\limits_{\partial\Omega}\big(u(x,t)-h\big)z(x)d_xS\nonumber\\
 &&=\int\limits_\Omega f\big(x,t,u(x,t),\nabla\varphi(x,t)\big)z(x)dx
\end{eqnarray}
for a.e. $t\in [0,T]$, where the null set in $[0,T]$ of those $t$ for which (\ref{2.13}) fails, does not depend on $z$. We integrate (\ref{2.13}) (with $s$ in place of $t$) over the interval $[0,t]$ ($0\le t\le T$) and integrate the first term on the left hand side by parts. Using the above notation $\widetilde{u}$ and (\ref{2.11}), we obtain
\begin{eqnarray}\label{2.14}
 &&\big\langle\widetilde{u}(t),z\big\rangle_{W^{1,q'}}\!+\!\int\limits_0^t\!\int\limits_\Omega\kappa\big(u(x,s)\big)\nabla\! u(x,s)\cdot\nabla z(x)dxds+g\int\limits_0^t\!\int\limits_{\partial\Omega}\big(u(x,s)\!-\!h\big)z(x)d_xSds\nonumber\\
 &&=\int\limits_\Omega u_0(x)z(x)dx+\int\limits_0^t\int\limits_\Omega f\big(x,s,u(x,s),\nabla\varphi(x,s)\big)z(x)dxds.
\end{eqnarray}

Let be $p=2$ and let be $f(x,t,u,\xi)=\sigma_0(u)|\xi|^2$ $\big(\big((x,t),u,\xi\big)\in Q_T\times\mathbb{R}\times\mathbb{R}^n$; cf. (\ref{1.13})$\big)$. Taking $z\equiv 1$ in (\ref{2.14}), we obtain
\[
\big\langle\widetilde{u}(t),1\big\rangle_{W^{1,q'}}+g\int\limits_0^t\int\limits_{\partial\Omega}\big(u(x,s)-h\big)d_xSds
=\int\limits_\Omega u_0(x)dx+\int\limits_0^t\int\limits_\Omega\bs{J}\cdot\bs{E}\,dxds,\quad t\in\,]0,T].
\]
\end{remark}

\section{Existence of weak solutions}\label{s3}

Our existence result for weak solutions of (\ref{1.1})--(\ref{1.5}) is the following

\begin{theo}
 Assume {\rm (H1)} and {\rm (H2)}. Suppose further that 
 \begin{equation}\label{3.1}
  \big(\sigma\big(u,|\xi|\big)\xi-\sigma\big(u,|\eta|\big)\eta\big)\cdot(\xi-\eta)\ge 0\quad\forall\; u\in\mathbb{R}, \;\forall\;\xi,\eta\in\mathbb{R}^n,
 \end{equation}
and
\begin{equation}\label{3.2}
\left\{ \begin{array}{l}
         f(x,t,u,\xi)=\alpha(x,t,u)\sigma\big(u,|\xi|\big)|\xi|^2\quad\forall\;\big((x,t),u,\xi\big)\in Q_T\times\mathbb{R}\times\mathbb{R}^n,\\[1mm]
         \text{\it where }\; \alpha:Q_T\times\mathbb{R}\to\mathbb{R}_+\;\text{ \it is Carath\'eodory},\\[1mm] 0\le\alpha(x,t,u)\le\alpha_0=\mathrm{const}\quad\forall\;\big((x,t),u\big)\in Q_T\times\mathbb{R},\\[1mm]
         \sigma=\sigma(u,\tau)\quad\text{\it as in {\rm (H1)}}.
        \end{array}\right.
\end{equation}

Let $\varphi_D$ and $u_0$ satisfy {\rm (\ref{2.2})} and {\rm (\ref{2.4})}, respectively, and suppose that
\begin{equation}\label{3.3}
 g=\mathrm{const}>0,\quad h=\mathrm{const}.
\end{equation}

Then there exists a pair
\[
 (\varphi,u)\in L^p\big(0,T; W^{1,p}(\Omega)\big)\times\Big(\bigcap\limits_{1<q<(n+2)/(n+1)}L^p\big(0,T;W^{1,q}(\Omega)\big)\Big)
\]
such that 
\begin{eqnarray}\label{3.4}
 &&\text{{\rm (2.5)} {\it and } {\rm (2.6)} {\it are satisfied}},\nonumber\\[1mm]
 &&\exists\;u'\in\bigcap\limits_{n+2<r<+\infty}L^1\big(0,T;\big(W^{1,r}(\Omega)\big)^*\big),
\end{eqnarray}
and for any $n+2<s<+\infty$ there holds
\begin{eqnarray}\label{3.5}
&& \left\{\begin{array}{l}
         {\displaystyle\int\limits_0^T\langle u',v\rangle_{W^{1,s}}dt+\int\limits_{Q_T}\kappa(u)\nabla u\cdot\nabla v\, dxdt+g\int\limits_0^T\int\limits_{\partial\Omega} (u-h)v\,d_xSdt}\\
         {\displaystyle=\int\limits_{Q_T}f(x,t,u,\nabla\varphi)v\,dxdt\quad\; \forall\;v\in L^\infty\big(0,T;W^{1,s}(\Omega)\big),}
        \end{array}\right.\\[1mm]
\label{3.6}
&& u(0)=u_0\quad\text{\it in }\; \big(W^{1,s}(\Omega)\big)^*.
\end{eqnarray}
Moreover, $u$ satisfies
\begin{eqnarray}\label{3.7}
&& \left\{\begin{array}{l}
         {\displaystyle\|u\|_{L^\infty(L^1)}+\lambda\int\limits_{Q_T}\frac{|\nabla u|^2}{\big(1+|u|\big)^{1+\lambda}}\,dxdt}\\[3mm]
         \le c\big(1+\|u_0\|_{L^1}+\big\|\,|\nabla\varphi_D|\,\big\|_{L^p}^p\big),\quad 0<\lambda<1\,\footnotemark
        \end{array}\right.\\[2mm]
\label{3.8}
&& u\in\bigcap\limits_{1<r<(n+2)/n}L^r\big(0,T;L^r(\Omega)\big).
\end{eqnarray}
\footnotetext{For notational simplicity, in what follows, for indexes we write $L^p(X)$ in place of $L^p(0,T;X)$. If there is no danger of confusion, we briefly write $L^p$ in place of $L^p(E)$ ($E\subset\mathbb{R}^m$).}
\end{theo}

The proof of this theorem is a further development of the approximation method we used in \cite{24}. In this paper, the function $\tau\mapsto\sigma(u,\tau)$ is assumed to satisfy the condition of strict monotonicity
\[
 \big(\sigma\big(u,|\xi|\big)\xi-\sigma\big(u|\eta|\big)\eta\big)\cdot(\xi-\eta)>0\quad\forall\;u\in\mathbb{R},\;\:\forall\; \xi,\eta\in\mathbb{R}^n,\;\:\xi\ne\eta.
\]
This condition allows to prove that the sequence $(\nabla\varphi_\varepsilon)_{\varepsilon>0}$ converges a.e. in $Q_T$ as $\varepsilon\to 0$, where $(\varphi_\varepsilon,u_\varepsilon)_{\varepsilon>0}$ is an approximate solution of the problem under consideration. Therefore, the discussion in \cite{24} includes the large class of source functions $f$ characterized by (H3).
\par
However, due to (\ref{3.1}), in the present paper we have to work only with the weak convergence of the sequence $(\varphi_\varepsilon)_{\varepsilon>0}$ in $L^q\big(0,T;W^{1,q}(\Omega)\big)$ as $\varepsilon\to 0$, which in turn makes the structure condition (\ref{3.2}) necessary for the passage to the limit $\varepsilon\to 0$.

\section{Proof of the theorem}\label{s4}

We begin by introducing two notations. For $\varepsilon>0$, define
\[
 f_\varepsilon(x,t,u,\xi)=\frac{f(x,t,u,\xi)}{1+\varepsilon f(x,t,u,\xi)},\quad \big((x,t),u,\xi\big)\in Q_T\times
 \mathbb{R}\times\mathbb{R}^n.
\]
To our knowledge, this approximation has been introduced for the first time by Bensoussan-Frehse in 1981. The function $f_\varepsilon$ is Carath\'eodory and satisfies the inequalities
\[
 0\le f_\varepsilon(x,t,u,\xi)\le\frac 1\varepsilon\quad\forall\;\big((x,t),u,\xi\big)\in Q_T\times\mathbb{R}\times\mathbb{R}^n.
\]

Let $(u_{0,\varepsilon})_{\varepsilon>0}$ be a sequence of functions in $L^2(\Omega)$ such that $u_{0,\varepsilon}\to u_0$ strongly in $L^1(\Omega)$ as $\varepsilon\to 0$.\hfill$\square$
\smallskip

\noindent
We divide the proof of the theorem into five steps.
\medskip

\noindent
$1^\circ$ {\it Existence of approximate solutions. }\, We have

\begin{lemma}\label{l1}
 For every $\varepsilon>0$ there exists a pair
 \[
  (\varphi_\varepsilon,u_\varepsilon)\in L^p\big(0,T;W^{1,p}(\Omega)\big)\times L^2\big(0,T;W^{1,2}(\Omega)\big)
 \]
such that
\begin{eqnarray}\label{4.1}
&& \left\{\begin{array}{l}
         {\displaystyle\varepsilon\int\limits_{Q_T}|\nabla\varphi_\varepsilon|^{p-2}\nabla\varphi_\varepsilon\cdot\nabla \zeta\,dxdt+\int\limits_{Q_T}\sigma\big(u_\varepsilon,|\nabla\varphi_\varepsilon|\big)\nabla\varphi_\varepsilon\cdot\nabla\zeta\,dxdt}\\[6mm]
         =0\quad\forall\;\zeta\in L^p\big(0,T;W_{\Gamma_D}^{1,p}(\Omega)\big)\,\footnotemark;
        \end{array}\right.\\[3mm]
\label{4.2}
&& \varphi_\varepsilon=\varphi_D\quad \text{a.e. on }\;\Sigma_D;  \\[2mm]
\label{4.3}
&&\exists\; u'_\varepsilon\in L^2\big(0,T;\big(W^{1,2}(\Omega)\big)^*\big);\\[2mm]
\label{4.4}
&&\left\{\begin{array}{l}
          {\displaystyle\int\limits_0^T\langle u'_\varepsilon,v\rangle_{W^{1,2}}dt+\int\limits_{Q_T}\kappa(u_\varepsilon)\nabla u_\varepsilon\cdot\nabla v\,dxdt+g\int\limits_0^T\int\limits_{\partial\Omega}(u_\varepsilon-h)v\,d_xSdt}\\
          {\displaystyle=\int\limits_{Q_T}f_\varepsilon(x,t,u_\varepsilon,\nabla\varphi_\varepsilon)v\,dxdt\quad\forall\; v\in L^2\big(0,T;W^{1,2}(\Omega)\big);}
         \end{array}\right.\\[3mm]
\label{4.5}
&& u_\varepsilon(0)=u_{0,\varepsilon}\quad\text{in }\; L^2(\Omega).
\end{eqnarray}
\footnotetext{If $1<p<2$, for $z\in W^{1,2}(\Omega)$ we define $\big|\nabla z(x)\big|^{p-2}\nabla z(x)=0$ a.e. in $\{x\in\Omega;\nabla z(x)=0\}$.}
\end{lemma}

\begin{proof}
 To begin with, we notice that, for all $\xi,\eta\in\mathbb{R}^n$,
\begin{eqnarray}\label{4.6}
&&\big(|\xi|^{p-2}\xi-|\eta|^{p-2}\eta\big)\cdot(\xi-\eta)\nonumber\\
&&\ge\left\{\begin{array}{l@{\qquad}l}
          {\displaystyle\frac{p-1}{\big(1+|\xi|+|\eta|\big)^{2-p}}\,|\xi-\eta|^2}&\text{if }\; 1<p\le 2,\\[4mm]
          {\displaystyle\min\Big\{\frac12,\frac1{2^{p-2}}\Big\}\,|\xi-\eta|^p}&\text{if }\;2<p<+\infty
         \end{array}\right.
\end{eqnarray}
(cf. \cite[pp. 71, 74]{21}, \cite{24}).
\par
For $\varepsilon>0$ and $(u,\tau)\in \mathbb{R}\times\mathbb{R}_+$, define
\[
 \begin{array}{l@{\qquad}l}
  \sigma_\varepsilon(u,0)=\sigma(u,0) &\text{if }\;\tau=0,\\[1mm]
  \sigma_\varepsilon(u,\tau)=\varepsilon\tau^{p-2}+\sigma(u,\tau)&\text{if }\; 0<\tau<+\infty.
 \end{array}
\]
Thus, by (\ref{3.1}) and (\ref{4.6}),
\[
 \big(\sigma_\varepsilon\big(u,|\xi|\big)\xi-\sigma_\varepsilon\big(u,|\eta|\big)\eta\big)\cdot(\xi-\eta)\ge\varepsilon\big(|\xi|^{p-2}\xi-|\eta|^{p-2}\eta\big)\cdot(\xi-\eta)>0
\]
for all $u\in\mathbb{R}$ and all $\xi,\eta\in\mathbb{R}^n$, $\xi\ne\eta$.
\par
The assertion of Lemma 1 now follows from \cite[Lemma 1]{24} with $\sigma_\varepsilon$ in place of $\sigma$. 
\end{proof}

\noindent
$2^\circ$ {\it A-priori estimates. }\, We have

\begin{lemma}\label{l2}
 Let be $(\varphi_\varepsilon,u_\varepsilon)$ as in Lemma~$1$. Then, for all $0<\varepsilon\le 1$,
 \begin{eqnarray}\label{4.7}
  &&\varepsilon\big\|\,|\nabla\varphi_\varepsilon|\,\big\|_{L^p}^p+\|\varphi_\varepsilon\|_{L^p(W^{1,p})}^p\le c\big(1+\big\|\,|\nabla\varphi_D|\,\big\|_{L^p}^p\big)\,\footnotemark;\\[2mm]
  \label{4.8}
  &&\left\{\begin{array}{l}
            {\displaystyle\|u_\varepsilon\|_{L^\infty(L^1)}+\lambda\int\limits_{Q_T}\frac{|\nabla u_\varepsilon|^2}{\big(1+|u_\varepsilon|\big)^{1+\lambda}}\,dxdt}\\[4mm]
            \le c\big(1+\|u_{0,\varepsilon}\|_{L^1}+\big\|\,|\nabla\varphi_D|\,\big\|_{L^p}^p\big),\quad 0<\lambda<1;
           \end{array}\right.
\end{eqnarray}
\footnotetext{Without any further reference, in what follows, by $c$ we denote constants which may change their numerical value from line to line, but do not depend on $\varepsilon$.}
\begin{eqnarray}\label{4.9}
 \|u_\varepsilon\|_{L^q(W^{1,q})}\!&\!\le\!&\!c\qquad\forall\;1<q<\frac{n+2}{n+1},\\[1mm]
 \label{4.10}
 \|u_\varepsilon\|_{L^r(L^r)}\!&\!\le\! &\!c\qquad\forall\; 1<r<\frac{n+2}n,\\[1mm]
 \label{4.11}
 \|u'_\varepsilon\|_{L^1((W^{1,q'})^*)}\!&\!\le\! &\!c\qquad\forall\; 1<q<\frac{n+2}{n+1}.
\end{eqnarray}
\end{lemma}

\begin{proof}
 By (\ref{4.2}), the function $\varphi_\varepsilon-\varphi_D$ is in $L^p\big(0,T;W_{\Gamma_D}^{1,p}(\Omega)\big)$. Inserting this function into (\ref{4.1}), we find
 \begin{eqnarray*}
  &&\varepsilon\int\limits_{Q_T}|\nabla\varphi_\varepsilon|^pdxdt+\int\limits_{Q_T}\sigma\big(u_\varepsilon,|\nabla\varphi_\varepsilon|\big)|\nabla\varphi_\varepsilon|^2dxdt\\
 && =\varepsilon\int\limits_{Q_T}|\nabla\varphi_\varepsilon|^{p-2}\nabla\varphi_\varepsilon\cdot\nabla\varphi_D\, dxdt+\int\limits_{Q_T}\sigma\big(u_\varepsilon,|\nabla\varphi_\varepsilon|\big)\nabla\varphi_\varepsilon\cdot\nabla\varphi_D\,dxdt.
 \end{eqnarray*}
From this, (\ref{4.7}) easily follows by combining (H1) and H\"older's inequality.
\par
Estimates (\ref{4.8})--(\ref{4.11}) can be proved by following line by line the proof of \cite[Lemma~2]{24}.
\end{proof}

\noindent
$3^\circ$ {\it Convergence of subsequences. }\, Let be $(\varphi_\varepsilon,u_\varepsilon)$ as in Lemma~1. From (\ref{4.7}) and (\ref{4.9}), (\ref{4.10}) we conclude that there exists a subsequence of $(\varphi_\varepsilon,u_\varepsilon)_{\varepsilon>0}$ (not relabelled) such that
\begin{equation}\label{4.12}
\varphi_\varepsilon\longrightarrow\varphi\quad\text{ weakly in }\;L^p\big(0,T;W^{1,p}(\Omega)\big)
\end{equation}
and
\begin{equation}\label{4.13}
 \left\{\begin{array}{l}
         {\displaystyle u_\varepsilon\to u\;\text{ weakly in }\; L^q\big(0,T;W^{1,q}(\Omega)\big)\quad\Big(1<q<\frac{n+2}{n+1}\Big)}\\[2mm]
         {\displaystyle \text{and weakly in }\; L^r\big(0,T;L^r(\Omega)\big)\quad\Big(1<r<\frac{n+2}n\Big)}         
        \end{array}\right.
\end{equation}
as $\varepsilon\to 0$. Then (\ref{4.2}) and (\ref{4.12}) yield $\varphi=\varphi_D$ a.e. on $\Sigma_D$, i.e., $\varphi$ satisfies (\ref{2.6}).
\par
Next, fix any $1<q<\frac{n+2}{n+1}$. Taking into account the embeddings (\ref{2.10}), from (\ref{4.9}) and (\ref{4.11}) we obtain by the aid of a well-known compactness result \cite[Prop.~1]{4} or \cite[Cor.~4]{26} the existence of a subsequence of $(u_\varepsilon)_{\varepsilon>0}$ (not relabelled) such that $u_\varepsilon\to u$ strongly in $L^q\big(0,T;L^2(\Omega)\big)$, and therefore
\begin{equation}\label{4.14}
 u_\varepsilon\longrightarrow u\quad\text{a.e. in }\; Q_T\;\text{ as }\; \varepsilon\longrightarrow 0.
\end{equation}

We prove estimate (\ref{3.7}). To begin with, we find an $0<\varepsilon_0\le 1$ such that
\[
 \|u_{0,\varepsilon}\|_{L^1}\le 1+\|u_0\|_{L^1}\quad\forall\;0<\varepsilon\le\varepsilon_0.
\]
Then, given any $\psi\in L^\infty(0,T)$, $\psi\ge 0$ a.e. in $[0,T]$, from (\ref{4.8}) it follows that
\begin{equation}\label{4.15}
 \int\limits_{Q_T}\big|u_\varepsilon(x,t)\psi(t)\big|dxdt\le C_0\int\limits_0^T\psi(t)dt\quad\forall\; 0<\varepsilon\le\varepsilon_0
\end{equation}
where
\[
 C_0:=c\big(1+\|u_0\|_{L^1}+\big\|\,|\nabla\varphi_D|\,\big\|_{L^p}^p\big).
\]
Taking the $\liminf\limits_{\varepsilon\to 0}$ in (\ref{4.15}), we find
\[
 \int\limits_{Q_T}\big|u(x,t)\psi(t)\big|dxdt\le C_0\int\limits_0^T\psi(t)dt.
\]
Hence,
\[
 \int\limits_\Omega\big|u(x,t)|dx\le C_0\quad\text{for a.e. }\; t\in[0,T].
\]

Next, from (\ref{4.8}) and (\ref{4.14}) we infer (by passing to a subsequence if necessary) that
\[
 \frac{\nabla u_\varepsilon}{\big(1+|u_\varepsilon|\big)^{(1+\lambda)/2}}\longrightarrow\frac{\nabla u}{\big(1+|u|\big)^{(1+\lambda)/2}}\quad \text{ weakly in }\;\big[L^2(Q_T)\big]^n
\]
as $\varepsilon\to 0$. Then taking the $\liminf\limits_{\varepsilon\to 0}$ in (\ref{4.8}) gives
\[
 \lambda\int\limits_{Q_T}\frac{|\nabla u|^2}{\big(1+
 |u|\big)^{1+\lambda}}\,dxdt\le C_0.
\]
\\[-5mm]
\text{}\hfill$\square$

\noindent
Summarizing, from (\ref{4.12})--(\ref{4.14}) we deduced the existence of a pair
\[
 (\varphi,u)\in L^p\big(0,T;W^{1,p}(\Omega)\big)\times\Big(\bigcap\limits_{1<q<(n+2)/(n+1)} L^q\big(0,T;W^{1,q}(\Omega)\big)\Big)
\]
which satisfies (\ref{2.6}) and (\ref{3.7}), (\ref{3.8}). It remains to prove that $(\varphi,u)$ satisfies the variational identity in (\ref{2.5}) and that (\ref{3.4})--(\ref{3.6}) hold true. This can be easily done by the aid of Lemma~3 and 4 we are going to prove next. 
\medskip

\noindent
$4^\circ$ {\it Passage to the limit $\varepsilon\to 0$. }\, We have

\begin{lemma}\label{l3}
 Let be $(\varphi_\varepsilon,u_\varepsilon)$ as in Lemma~$1$, and let be $(\varphi,u)$ as in {\rm (\ref{4.12})}, {\rm (\ref{4.13})}. Then
\begin{equation}\label{4.16}
 \int\limits_{Q_T}\sigma\big(u,|\nabla\varphi|\big)\nabla\varphi\cdot\nabla\zeta\,dxdt=0\quad\forall\;\zeta\in L^p\big(0,T;W_{\Gamma_D}^{1,p}(\Omega)\big)
\end{equation}
i.e., $(\varphi,u)$ satisfies {\rm (\ref{2.5})};
\begin{eqnarray}\label{4.17}
 && \sigma\big(u_\varepsilon,|\nabla\varphi_\varepsilon|\big)\nabla\varphi_\varepsilon\longrightarrow\sigma\big(u,|\nabla\varphi|\big)\nabla\varphi\;\text{ weakly in }\; \big[L^{p'}(Q_T)\big]^n\;\text{ as }\;\varepsilon\longrightarrow 0;
 \\[2mm]
 \label{4.18}
 &&\sigma \big(u_\varepsilon,|\nabla\varphi_\varepsilon|\big)|\nabla\varphi_\varepsilon|^2\longrightarrow\sigma\big(u,|\nabla\varphi|\big)|\nabla\varphi|^2\;\text{ weakly in }\; L^1(Q_T)\;\text{ as }\;\varepsilon\longrightarrow 0.        
\end{eqnarray}
\end{lemma}

\noindent
{\it Proof of\/} (\ref{4.16}) (cf. the ``monotonicity trick'' in \cite[pp. 161, 172]{22}, \cite[p. 474]{30}). The function
$\varphi_\varepsilon-\varphi_D$ is in $L^p\big(0,T;W_{\Gamma_D}^{1,p}(\Omega)\big)$ (see (\ref{4.2})). Thus, given any $\psi\in L^p\big(0,T;W_{\Gamma_D}^{1,p}(\Omega)\big)$, the function $\zeta=\varphi_\varepsilon-\varphi_D-\psi$ is admissible in (\ref{4.1}). By the monotonicity condition (\ref{3.1}) ($\xi=\nabla\varphi_\varepsilon$ and $\eta=\nabla(\psi+\varphi_D)$),
\begin{eqnarray*}
&&0=\varepsilon\int\limits_{Q_T}|\nabla\varphi_\varepsilon|^{p-2}\nabla\varphi_\varepsilon\cdot\nabla\big(\varphi_\varepsilon-(\psi+\varphi_D)\big)dxdt\\
&&\qquad 
+\int\limits_{Q_T}\sigma\big(u_\varepsilon,|\nabla\varphi_\varepsilon|\big)\nabla\varphi_\varepsilon\cdot\nabla\big(\varphi_\varepsilon-(\psi+\varphi_D)\big)dxdt\\
&&\ge-\varepsilon\int\limits_{Q_T}|\nabla\varphi_\varepsilon|^{p-2}\nabla\varphi_\varepsilon\cdot\nabla(\psi+\varphi_D)dxdt\\
&&\qquad+\int\limits_{Q_T}\sigma\big(u_\varepsilon,\big|\nabla(\psi+\varphi_D)\big|\big)\nabla(\psi+\varphi_D)\cdot\nabla\big(\varphi_\varepsilon-(\psi+\varphi_D)\big)dxdt.
\end{eqnarray*}
The passage to the limit $\varepsilon\to 0$ gives
\begin{equation}\label{4.19}
0\ge\int\limits_{Q_T}\sigma\big(u,\big|\nabla(\psi+\varphi_D)\big|\big)\nabla(\psi+\varphi_D)\cdot\nabla\big(\varphi-(\psi+\varphi_D)\big)dxdt 
\end{equation}
(cf. (\ref{4.7}), (\ref{4.12}) and (\ref{4.14})).
\par
Let $\zeta\in L^p\big(0,T;W_{\Gamma_D}^{1,p}(\Omega)\big)$. For any $\lambda>0$, we insert $\psi=\varphi-\varphi_D\mp\lambda\zeta$ into (\ref{4.19}), divide then by $\lambda$ and carry through the passage to the limit $\lambda\to 0$. It follows
\[
 \int\limits_{Q_T}\sigma\big(u,|\nabla\varphi|\big)\nabla\varphi\cdot\nabla\zeta\,dxdt=0.
\]
\smallskip

\noindent
{\it Proof of\/} (\ref{4.17}). From (H1) and (\ref{4.7}) it follows that there exists a subsequence of $(\nabla\varphi_\varepsilon)_{\varepsilon>0}$ (not relabelled) such that
\[
 \sigma\big(u_\varepsilon,|\nabla\varphi_\varepsilon|\big)\nabla\varphi_\varepsilon\longrightarrow\bs{F}\quad\text{weakly in }\;\big[L^{p'}(Q_T)\big]^n \;\text{as }\;\varepsilon\longrightarrow 0. 
\]
The function $\zeta=\varphi-\varphi_D$ being admissible in (\ref{4.1}), we find
\[
 \int\limits_{Q_T}\bs{F}\cdot\nabla(\varphi-\varphi_D)dxdt=\lim\limits_{\varepsilon\to 0}\int\limits_{Q_T}\sigma\big(u_\varepsilon,|\nabla\varphi_\varepsilon|\big)\nabla\varphi_\varepsilon\cdot\nabla(\varphi-\varphi_D)dxdt=0.
\]
Thus, using (\ref{4.1}) with $\zeta=\varphi_\varepsilon-\varphi_D$, it follows
\begin{eqnarray}\label{4.20}
 \int\limits_{Q_T}\bs{F}\cdot\nabla\varphi\,dxdt&\!\!=\!\!&\int\limits_{Q_T}\bs{F}\cdot\nabla\varphi_D\,dxdt\nonumber\\
 &=&\lim\limits_{\varepsilon\to 0}\int\limits_{Q_T}\sigma\big(u_\varepsilon,|\nabla\varphi_\varepsilon|\big)\nabla\varphi_\varepsilon\cdot \nabla\varphi_D\,dxdt\nonumber\\
 &\ge&\liminf\limits_{\varepsilon\to 0}\int\limits_{Q_T}\sigma\big(u_\varepsilon,|\nabla\varphi_\varepsilon|\big)|\nabla\varphi_\varepsilon|^2dxdt.
\end{eqnarray}

Claim (\ref{4.17}) is now easily seen by the aid of the ``monotonicity trick'' with respect to the dual pairing $\big(\big[L^p(Q_T)\big]^n,\big[L^{p'}(Q_T)\big]^n\big)$. Indeed, let $\bs{G}\in \big[L^p(Q_T)\big]^n$. Using (\ref{3.1}) with $\xi=\bs{G}$, $\eta=\nabla\varphi_\varepsilon$, we find by the aid of (\ref{4.12}), (\ref{4.20}) and Lebesgue's Dominated Convergence Theorem
\[
 \int\limits_{Q_T}\sigma\big(u,|\bs{G}|\big)\bs{G}\cdot(\bs{G}-\nabla\varphi)dxdt\ge\int\limits_{Q_T}\bs{F}\cdot(\bs{G}-\nabla\varphi)dxdt.
\]
Hence, given $\bs{H}\in \big[L^p(Q_T)\big]^n$ and $\lambda>0$, we take $\bs{G}=\nabla\varphi\pm \lambda\bs{H}$, divide by $\lambda>0$ and carry through the passage to the limit $\lambda\to 0$ to obtain
\[
 \int\limits_{Q_T}\sigma\big(u,|\nabla\varphi|\big)\nabla\varphi\cdot\bs{H}\,dxdt=\int\limits_{Q_T}\bs{F}\cdot\bs{H}\, dxdt.
\]
Whence (\ref{4.17}).
\bigskip

\noindent
{\it Proof of\/} (\ref{4.18}). Define
\[
 g_\varepsilon=\big(\sigma\big(u_\varepsilon,|\nabla\varphi_\varepsilon|\big)\nabla\varphi_\varepsilon-\sigma\big(u_\varepsilon,|\nabla\varphi|\big)\nabla\varphi\big)\cdot\nabla(\varphi_\varepsilon-\varphi)\quad\text{a.e. in }\; Q_T.
\]
By the aid of (\ref{4.17}), (\ref{4.16}) and $u_\varepsilon\to u$ a.e. in $Q_T$ (see (\ref{4.14})) one easily obtains
\[
 \lim\limits_{\varepsilon\to 0}\int\limits_{Q_T}g_\varepsilon\, dxdt=0.
\]
By (\ref{3.1}), $g_\varepsilon\ge 0$ a.e. in $Q_T$. Thus

\begin{equation}\label{4.21}
 \lim\limits_{\varepsilon\to 0}\int\limits_{Q_T} g_\varepsilon z\,dxdt=0\quad\forall\; z\in L^\infty(Q_T).
\end{equation}
We next multiply each term of the equation
\[
 \sigma\big(u_\varepsilon,|\nabla\varphi_\varepsilon|\big)|\nabla\varphi_\varepsilon|^2=g_\varepsilon+\sigma\big(u_\varepsilon,|\nabla\varphi_\varepsilon|\big)\nabla\varphi_\varepsilon\cdot\nabla\varphi+\sigma\big(u_\varepsilon,|\nabla\varphi|\big)\nabla\varphi\cdot\nabla(\varphi_\varepsilon-\varphi)
\]
by $z\in L^\infty(Q_T)$ and integrate over $Q_T$. Then (\ref{4.18}) follows from (\ref{4.21}), (\ref{4.17}) and (\ref{4.14}), (\ref{4.12}).\hfill$\square$
\smallskip

\noindent
The next lemma is fundamental to the passage to the limit $\varepsilon\to 0$ in (\ref{4.4}).

\begin{lemma}\label{l4}
 Let be $(\varphi_\varepsilon,u_\varepsilon)$ as in Lemma~$1$, and let be $(\varphi,u)$ as in {\rm (\ref{4.12}), (\ref{4.13})}. Then, for any $z\in L^\infty(Q_T)$,
 \begin{equation}\label{4.22}
  \lim\limits_{\varepsilon\to 0}\int\limits_{Q_T}f_\varepsilon(x,t,u_\varepsilon,\nabla\varphi_\varepsilon)z\,dxdt=\int\limits_{Q_T}f(x,t,u,\nabla\varphi)z\,dxdt.
 \end{equation}
\end{lemma}

\begin{proof}
 For notational simplicity, we write $(\cdot,\cdot)$ in place of the variables $(x,t)$.
 \par
 The structure condition (\ref{3.2}) and the definition of $f_\varepsilon$ yield
 \[
  \int\limits_{Q_T}\frac{f(\cdot,\cdot,u_\varepsilon,\nabla\varphi_\varepsilon)}{1+\varepsilon f(\cdot,\cdot,u_\varepsilon,\nabla \varphi_\varepsilon)}\,z\,dxdt-\int\limits_{Q_T}f(\cdot,\cdot,u,\nabla\varphi)z\,dxdt=J_{1,\varepsilon}+ J_{2,\varepsilon}+J_{3,\varepsilon}
 \]
where
\begin{eqnarray*}
 J_{1,\varepsilon}&\!\!=\!\!&\int\limits_{Q_T} A_\varepsilon B_\varepsilon dxdt,\\
A_\varepsilon&\!\!=\!\!&z\alpha(\cdot,\cdot,u_\varepsilon)\Big(\frac1{1+\varepsilon\alpha(\cdot,\cdot,u_\varepsilon)\sigma\big(u_\varepsilon,|\nabla\varphi_\varepsilon|\big)|\nabla\varphi_\varepsilon|^2}-1\Big)\\[1mm]
B_\varepsilon&\!\!=\!\!&\sigma\big(u_\varepsilon,|\nabla\varphi_\varepsilon|\big)|\nabla\varphi_\varepsilon|^2,
\end{eqnarray*}
and
\begin{eqnarray*}
 J_{2,\varepsilon}&\!\!=\!\!&\int\limits_{Q_T}z\big(\alpha(\cdot,\cdot,u_\varepsilon)-\alpha(\cdot,\cdot,u)\big)B_\varepsilon\, dxdt,\\
 J_{3,\varepsilon}&\!\!=\!\!&\int\limits_{Q_T}z\alpha(\cdot,\cdot,u)\big(B_\varepsilon-\sigma\big(u,|\nabla\varphi|\big)|\nabla\varphi|^2\big)dxdt.
\end{eqnarray*}
Observing that $0\le\alpha\le\alpha_0=\mathrm{const}$ a.e. in $Q_T$ (see (\ref{3.2})), we find
\begin{equation}\label{4.23}
 |A_\varepsilon|\le\alpha_0\|z\|_{L^\infty}\quad\text{a.e. in }\; Q_T,\quad\forall\;\varepsilon>0.
\end{equation}
On the other hand, from 
\[
 \int\limits_{Q_T}\alpha(\cdot,\cdot,u_\varepsilon)\sigma\big(u_\varepsilon,|\nabla\varphi_\varepsilon|\big)|\nabla\varphi_\varepsilon|^2dxdt\le c\quad\forall\;\varepsilon>0
\]
it follows (by going to a subsequence if necessary) that 
\[
 \varepsilon\alpha(\cdot,\cdot,u_\varepsilon)\sigma\big(u_\varepsilon,|\nabla\varphi_\varepsilon|\big)|\nabla\varphi_\varepsilon|^2\longrightarrow 0\quad\text{a.e. in }\; Q_T \;\text{ as }\;\varepsilon\longrightarrow 0.
\]
Hence,
\begin{equation}\label{4.24}
 A_\varepsilon\longrightarrow 0\quad\text{a.e. in }\; Q_T \;\text{ as }\; \varepsilon\longrightarrow 0.
\end{equation}

From (\ref{4.23}), (\ref{4.24}) and $B_\varepsilon\to\sigma\big(u,|\nabla\varphi|\big)|\nabla\varphi|^2$ weakly in $L^1(Q_T)$ (see (\ref{4.18})) we conclude with the help of Egorov's theorem and the absolute continuity of the intgral that
\[
 J_{1,\varepsilon}=\int\limits_{Q_T}A_\varepsilon B_\varepsilon\, dxdt\longrightarrow 0\quad\text{as }\;\varepsilon\longrightarrow 0
\]
(see, e.g., \cite[p. 54, Prop. 1\,(i)]{12}). Analogously,
\[
 J_{k,\varepsilon}\longrightarrow 0\quad\text{as }\;\varepsilon\longrightarrow 0\quad (k=2,3).
\]
Whence (\ref{4.22}).
\end{proof}

\noindent
$5^\circ$ {\it Proof of\/} (\ref{3.4})--(\ref{3.6}). \, Let $n+2<r<+\infty$ (i.e., setting $q=r'$, then $1<q<\frac{n+2}{n+1}$, $q'=r$, and vice versa).
\par
Let be $z\in W^{1,r}(\Omega)$ and $\psi\in C^1\big([0,T]\big)$, $\psi(T)=0$. We set $v(x,t)=z(x)\psi(t)$ for a.e. $(x,t)\in Q_T$. An integration by parts gives
\begin{eqnarray*}
 \int\limits_0^T\langle u'_\varepsilon,v\rangle_{W^{1,2}}dt\!&\!=\!&\!-\big\langle u_\varepsilon(0),z\big\rangle_{W^{1,2}}\psi(0)-\int\limits_0^T\langle z\psi',u_\varepsilon\rangle_{W^{1,2}}dt\\
 \!&\!=\!&\!-\int\limits_\Omega u_\varepsilon(\cdot,0)z\,dx\,\psi(0)-\int\limits_{Q_T}u_\varepsilon z\psi'dxdt\qquad \text{[by (\ref{2.1})]}
\end{eqnarray*}
(see \cite[p. 54, Prop. 2.5.2 with $p=q=2$, $r=1$ therein]{10}).
\par
With the help of (\ref{4.13}), (\ref{4.14}) and (\ref{4.22}) the passage to the limit $\varepsilon\to 0$ in (\ref{4.4}) (with $v=z\psi$ therein) is easily done. We find
\begin{eqnarray}\label{4.25}
&& -\int\limits_{Q_T} uz\psi'dxdt+\int\limits_{Q_T}\kappa(u)\nabla u\cdot\nabla z\psi\,dxdt+g\int\limits_0^T\int\limits_{\partial\Omega}(u-h)z\psi \,d_xSdt\nonumber\\
&&=\int\limits_\Omega u_0z\,dx\psi(0)+\int\limits_{Q_T} f(x,t,u,\nabla \varphi)z\psi\,dxdt
\end{eqnarray}
(recall $u_\varepsilon(\cdot,0)=u_{0,\varepsilon}\to u_\varepsilon$ strongly in $L^1(\Omega)$). Following line by line the arguments in \cite{24}, from (\ref{4.25}) we deduce the existence of the distributional derivative
\[
 u'\in L^1\big(0,T;\big(W^{1,r}(\Omega)\big)^*\big)
\]
(cf. \cite[p. 154, Prop. A6]{3}), i.e., (\ref{3.4}) holds. Moreover, we have
\begin{equation}\label{4.26}
\int\limits_0^T\big\langle u'(t),z\psi(t)\big\rangle_{W^{1,r}}dt+\big\langle\widetilde{u}(0),z\big\rangle_{W^{1,r}}\psi(0)=-\int\limits_{Q_T} uz\psi'dxdt\qquad \text{ [by (\ref{2.1})]}, 
\end{equation}
where $\widetilde{u}\in C\big([0,T];\big(W^{1,r}(\Omega)\big)^*\big)$ is as in Section 2 (see \cite[p. 54, Prop. 2.5.2 with $p=1$, $q=+\infty$, $r=1$ therein]{10}). We insert (\ref{4.26}) into (\ref{4.25}) and obtain
\begin{eqnarray}\label{4.27}
 &&\int\limits_0^T\big\langle u'(t),z\psi(t)\big\rangle_{W^{1,r}}dt + \big\langle\widetilde{u}(0),z\big\rangle_{W^{1,r}}\psi(0)\nonumber\\
 &&+\int\limits_{Q_T}\kappa(u)\nabla u\cdot\nabla z\psi\,dxdt+g\int\limits_0^T\int\limits_{\partial\Omega}(u-h)z\psi\, d_xSdt\nonumber\\
 &&= \int\limits_\Omega u_0z\,dx\,\psi(0)+\int\limits_{Q_T} f(x,t,u,\nabla\varphi)z\psi\,dxdt
\end{eqnarray}
for all $z\in W^{1,r}(\Omega)$ and all $\psi\in C^1\big([0,T]\big)$, $\psi(T)=0$. 
\par
To prove (\ref{3.5}), we take $\psi\in C_c^1\big(\,]\,0,T\,[\,\big)$ in (\ref{4.27}). A routine argument yields
\begin{eqnarray}\label{4.28}
 &&\big\langle u'(t),z\big\rangle_{W^{1,r}}+\int\limits_\Omega\kappa(u)\nabla u\cdot\nabla z\,dx+g\int\limits_{\partial\Omega}(u-h)z\,d_xS\nonumber\\
 &&=\int\limits_\Omega f(x,t,u,\nabla\varphi)z\,dx
\end{eqnarray}
for all $z\in W^{1,r}(\Omega)$ and a.e. $t\in [0,T]$, where the null set in $[0,T]$ of those $t$ for which (\ref{4.28}) fails, does not depend on $z$. Now, given $v\in L^\infty\big(0,T;W^{1,s}(\Omega)\big)$ ($n+2<s<+\infty$), we insert $z=v(\cdot,t)$ into (\ref{4.28}) (with $r=s$ therein) and integrate over the interval $[0,T]$. Whence (\ref{3.5}).
\par
Equ. (\ref{3.6}) in $\big(W^{1,s}(\Omega)\big)^*$ is now easily seen. Indeed, let $z\in W^{1,s}(\Omega)$ ($n+2<s<+\infty$), and let $\psi\in C^1\big([0,T]\big)$, $\psi(0)=1$ and $\psi(T)=0$. We multiply (\ref{4.28}) by $\psi(t)$ and integrate over $[0,T]$. Combining (\ref{4.27}) and (\ref{4.28}), we obtain
\[
 \big\langle \widetilde{u}(0),z\big\rangle_{W^{1,s}}=\int\limits_\Omega u_0z\,dx,
\]
i.e., (\ref{3.6}) holds (cf. (\ref{2.11}) with $q'=s$ therein).
\par
The proof of the theorem is complete. 

\end{document}